\documentclass[11pt]{amsart}

\usepackage{amssymb}

\usepackage{enumerate}

\usepackage{graphicx}

\makeatletter
\@namedef{subjclassname@2010}{%
  \textup{2010} Mathematics Subject Classification}
\makeatother

\newtheorem{thm}{Theorem}[section]
\newtheorem{cor}[thm]{Corollary}
\newtheorem{lem}[thm]{Lemma}
\newtheorem{prop}[thm]{Proposition}

\theoremstyle{definition}

\newtheorem{exa}[thm]{Example}

\newtheorem{xrem}{Remark}

\numberwithin{equation}{section}

\begin{document}


\baselineskip=17pt



\title[Manjul Gupta, Aneesh Mundayadan]{$q$-Frequent hypercyclicity in spaces of operators}

\author[Manjul Gupta]{Manjul Gupta}
\address{Department of Mathematics and Statistics\\ Indian Institute of Technology Kanpur\\
208 016 Kanpur, India} \email{manjul@iitk.ac.in}

\author[Aneesh M]{Aneesh Mundayadan}
\address{Department of Mathematics and Statistics\\ Indian Institute of Technology Kanpur\\
208 016 Kanpur, India} \email{aneeshm@iitk.ac.in}

\date{19.02.2016}

\begin{abstract}
We provide conditions for a linear map of the form $C_{R,T}(S)=RST$
to be $q$-frequently hypercyclic on algebras of operators on separable
Banach spaces. In particular, if $R$ is a bounded operator
satisfying the $q$-Frequent Hypercyclicity Criterion, then the map
$C_{R}(S)$=$RSR^*$ is shown to be $q$-frequently hypercyclic on the
space $\mathcal{K}(H)$ of all compact operators and the real
topological vector space $\mathcal{S}(H)$ of all self-adjoint
operators on a separable Hilbert space $H$. Further we provide a
condition for $C_{R,T}$ to be $q$-frequently hypercyclic on the
Schatten von Neumann classes $S_p(H)$. We also characterize
frequent hypercyclicity of $C_{M^*_\varphi,M_\psi}$ on the
trace-class of the Hardy space, where the symbol $M_\varphi$ denotes
the multiplication operator associated to $\varphi$.

\end{abstract}



\maketitle

\section{Introduction}

This paper is a continuation of our earlier work \cite{M-M} on
$q$-frequent hypercyclicity, which coincides with frequent
hypercyclicity for $q=1$. We study here this concept for linear maps
defined on Banach algebras of operators on Banach and Hilbert
spaces. Hypercyclicity in spaces of operators was initiated by K.C. Chan \cite{C} and subsequently studied by J. Bonet, F. Martinez-Gimenez and A. Peris \cite{BFP}, K.C. Chan and R. Taylor \cite{CT}, F. Martinez-Gimenez and A. Peris \cite{FP}
and H. Petersson \cite{H}. Indeed, left multiplication operators $\mathfrak{L}_R(S)=RS$
were considered in \cite{BFP},\cite{C},\cite{CT} and \cite{FP} and their
general form $C_{R,T}(S)=RST$ was studied in \cite{H}. A collective work in \cite{BFP},\cite {C} and \cite{CT} states that a bounded operator
$R$ on a separable Banach space $X$ satisfies the Hypercyclicity
Criterion if and only if the left multiplication operator $\mathfrak{L}_R$ is hypercyclic on
$\mathcal{L}(X)$ in the topology of pointwise convergence.
This result holds for the topology of uniform convergence on compact subsets if $X^*$ is separable and $X$ has the
approximation property, see \cite{BFP}. In \cite{H} H. Petersson proved that if $T$ satisfies the Hypercyclicity Criterion in a separable Hilbert space, then $\mathfrak{L}_T$ as well as the conjugate operator $C_T$ is hypercyclic on the Schatten von Neumann classes $S_p(H)$, $1\leq p<\infty$ and $\mathcal{K}(H)$.

In Section 3 we provide a sufficient criterion for $C_{R,T}$ to be $q$-frequently
hypercyclic on the algebra of compact operators on Banach spaces and
give applications to the unilateral and bilateral shift operators, and
in Section 4 we continue the study in the space $S_p(H)$. Finally in Section 5, using
an Eigenvalue Criterion, we characterize frequent hypercyclicity of certain maps of the form
$C_{R,T}$ defined on spaces of operators on the classical Hardy space and $\ell^p$.

\section{Preliminaries}

A continuous operator $T$ on a topological vector space (TVS) $X$ is said to be
$\mathbf{hypercyclic}$ if the set $\{T^nx: n\geq1\}$ is dense in $X$
for some $x\in X$. For $q\in \mathbb{N}$ (\textit{the set of natural
numbers}), $T$ is said to be $q$-$\mathbf{frequently}$
$\mathbf{hypercyclic}$ (see \cite{M-M}) if there exists a vector
$x\in X$ such that the set $\{n\in \mathbb{N}: T^nx\in U\}$ has
positive $q$-lower density for each non-empty open set $U\subset X$,
where the $q$-lower density of $A\subset \mathbb{N}$ is defined as
\begin{center}
$\displaystyle \underline{\text{$q$-dens}}(A)=\liminf_{N\rightarrow\infty}$ $\frac
{\text{card}\{n\in A:~ n\leq N^q\}}{N}$.
\end{center}
For $q=1$, the above  notion is known as frequent hypercyclicity of
an operator, studied in \cite{BG}, \cite{BGE} and \cite{BE}. If $T$
is frequently hypercyclic, then it is $q$-frequently hypercyclic for
all $q\in \mathbb{N}$, however, the converse is not true, cf.
\cite{M-M}.


Let $X$ and $Y$ be separable Banach spaces. The space of all
bounded (resp. of all compact) operators from $X$ to $Y$ is denoted
by $\mathcal{L}(X,Y)$ (resp. $\mathcal{K}(X,Y)$). We shall use the
symbols $\mathcal{L}(X)$ and $\mathcal{K}(X)$ for $\mathcal{L}(X,X)$
and $\mathcal{K}(X,X)$ respectively. The real subspace of
$\mathcal{L}(H)$, of all self-adjoint operators on a separable
infinite dimensional Hilbert space $H$ is denoted by
$\mathcal{S}(H)$ and is equipped with the topology of uniform
convergence on compact subsets (COT). Also, for $p\in [1,\infty)$
the Schatten von Neumann class $S_p(H)$ is defined as the space
of all operators $T\in \mathcal{L}(H)$ for which the approximation
numbers $(a_n(T))\in \ell^p$, where
\begin{center}
$a_n(T)=\inf \{\|T-F\|:~ rank(F)<n\}, n\geq 1$,
\end{center}
See \cite{DJT} and \cite{BST} for more details on the Schatten
von Neumann classes.

For $R\in \mathcal{L}(X)$, the left and right multiplication
operators are respectively defined as $\mathfrak{L}_R(S)=RS$ and
$\mathfrak{R}_R(S)=SR$. Also, if $R$ is a bounded operator on a
Hilbert space, then the conjugate operator $C_R$ is defined as
$C_R(S)=RSR^*$.

Recall that a Banach space $X$ is said to have the \textbf{approximation
property (AP)} if the identity operator on $X$ can be approximated by
finite rank operators uniformly on compact subsets of $X$; that is,
for any $\epsilon >0$ and $K\subset X$ compact, there exists an
operator $F$ of finite rank such that $\|F(x)-x\|<\epsilon$,
$\forall x\in K$. If $X$ has the AP, then finite rank operators are
norm-dense in $\mathcal{K}(Y,X)$ for all Banach spaces $Y$, cf. \cite{DJT} and \cite{LT}.

A series $\sum_{n\geq 1} x_{n,j}$ in an $F$-space is said to be \textbf{unconditionally convergent uniformly} in $j\geq 0$ if for every $\delta >0$, there exists $N\in \mathbb{N}$ such that $\|\sum_{n\in F} x_{n,j}\|<\delta$ for all finite sets $F\subset [N,\infty)$ and $j$. We will make use of the following inequality in our subsequent work: let $(\lambda_n)$ be a scalar sequence and $\sum_{n\geq 1} x_n$ a series in a Banach space. Then if $F\subset \mathbb{N}$ is finite, we have
\begin{equation}\label{qq}
\displaystyle \left \|\sum_{n\in F} \lambda_n x_n \right \| \leq 4~\sup_{n\in F}|\lambda_n| \sup_{G\subseteq F}\left \|\sum_{n\in G} x_n \right \|,
\end{equation}
cf. \cite{Oho} (See also \cite{KG}).

\section{$q$-Frequent Hypercyclicity in $\mathcal{K}(X)$,
$\mathcal{L}(X)$ and $\mathcal{S}(H)$}

In this section, we first obtain a sufficient criterion for the
$q$-frequent hypercyclicity of $C_{R,T}$ on the Banach algebras
$\mathcal{L}(X,Y)$ and $\mathcal{K}(X,Y)$, where $R\in
\mathcal{L}(Y)$ and $T\in \mathcal{L}(X)$. The next result is already known for $q=1$, cf. \cite{PE}, Remark 9.10. However, following the proof of the frequent hypercyclicity criterion given in Theorem 6.18 of \cite{BM}, we outline the proof for a given $q\in \mathbb{N}$.
\begin{thm} \textbf{$(q$-$FHC$ $Criterion)$} \label{1.3} Let $X$ be a separable $F$-space and $D$ be a dense set in $X$. If for each $x\in D$,
there exists a sequence $(x_n)_{n\geq 0}$ in $X$ such that $x_0=x$ and
\begin{itemize}
\item[(a)] $\sum_{n=0}^{r} T^{r^q-(r-n)^q}(x)$ converges unconditionally, uniformly in $r\geq 0$,
\item[(b)] $\sum_{n\geq0} x_{(n+r)^q-r^q}$ converges unconditionally, uniformly in $r\geq 0$; and
\item[(c)] $T^{n^q}x_{n^q}=x$, $T^{n^q}x_{m^q}=x_{m^q-n^q}$ for $m>n\geq 0$,
\end{itemize}
then $T$ is $q$-frequently hypercyclic on $X$.
\end{thm}
\begin{proof}
Without loss of generality, one may assume that $D$ is countable. Write $D=\{x_k:k\in \mathbb{N}\}$ and fix $(\epsilon_k)$ such that $k\epsilon_k+\sum_{j\geq k+1}\epsilon_j \rightarrow 0$. For each $x_k$, there exists a sequence $(x_{n,k})_{n\geq 0}$ with the conditions in the hypotheses being satisfied. By (a) and (b), it is possible to find an increasing sequence $(N_k)$ of natural numbers such that $\|\sum_{n\in F} T^{r^q-(r-n)^q}(x_i)\| <\epsilon_k$ for $F\subset [N_k,\infty)\cap \{1,...,r\}$, and $\|\sum_{n\in G} x_{(n+r)^q-r^q,i}\|<\epsilon_k$ for $G\subset [N_k,\infty)$, uniformly in $r\geq 0$, where $1\leq i\leq k$. By Lemma 6.19 of \cite{BM}, corresponding to $(N_k)$, we find a sequence $(J_k)$ of subsets of $\mathbb{N}$ such that $\underline{dens}(J_k)>0$, $\min (J_k)\geq k$ and $|m-n|\geq N_k+N_j$ for all $m\in J_k$, $n\in J_j$ and $m\neq n$. With these properties, the vector $x=\sum_{\ell \geq 1}\sum_{n\in J_\ell} x_{n^q,\ell}$ is a frequently hypercyclic vector for the sequence $(T^{n^q})$, and thus it is a $q$-frequently hypercyclic vector for $T$.
\end{proof}

Let $y\otimes x^*$ be the one-rank operator $x\rightarrow x^*(x)y$, where $y\in Y$
and $x^*\in X^*$ and $(\mathcal{I}(X,Y),\|.\|_{\mathcal{I}(X,Y)})$ be a
Banach space of operators from $X$ to $Y$ such that the set of
all finite-rank operators is $\|.\|_{\mathcal{I}(X,Y)}$-dense in
$\mathcal{I}(X,Y)$ and $\|y\otimes x^*\|_{\mathcal{I}(X,Y)}=\|y\|\|x^*\|$ for all $y\in Y$ and $x^*\in
X^*$. We have the following result concerning the
separability of $\mathcal{L}(X,Y)$ with respect to the topologies
SOT and COT, and of $\mathcal{K}(X,Y)$ in the operator norm topology.

\begin{prop}\label{P}
Let $X$ and $Y$ be separable Banach spaces. Then the following are true.\\
$(1)$ If $D$ is a countable dense subset of $Y$ and $\Phi$ is a
countable weak$^*$-dense subset of $X^*$, then the set
\begin{equation*}\label{gdphi}
\displaystyle \mathcal{G}_{D,\Phi}=\left \{\sum_{n=1}^{N}y_n\otimes
x_n^*: y_n\in D, x_n^*\in \Phi, N\in \mathbb{N} \right\}
\end{equation*}
is a countable SOT-dense subset of
$\mathcal{L}(X,Y)$.\\
$(2)$ If $X^*$ is separable and $\Phi$ is norm-dense, then the above
set $\mathcal{G}_{D,\Phi}$ is $\|.\|_{\mathcal{I}(X,Y)}$-dense in
$\mathcal{I}(X,Y)$.\\
$(3)$ Suppose that $X^*$ is separable and $Y$ has the AP. If $\Phi$
is norm-dense, then $\mathcal{G}_{D,\Phi}$ is norm-dense in
$\mathcal{K}(X,Y)$ and COT-dense in $\mathcal{L}(X,Y)$.
\end{prop}
\begin{proof}

The proof of $(1)$ is similar to the case of $X=Y$, proved in \cite{PE}, p. 277. Further, by the properties of $\mathcal{I}(X,Y)$ mentioned above, part $(2)$ follows since we can approximate every operator of finite rank by elements of $\mathcal{G}_{D,\Phi}$ in the norm $\|.\|_{\mathcal{I}(X,Y)}$. To get part $(3)$, let us assume that $Y$ has the AP. Then the space $\mathcal{F}(X,Y)$ of all finite-rank
operators is norm-dense in $\mathcal{K}(X,Y)$ for every Banach space $X$. Moreover, $\|v\otimes u^*\|_{op}=\|v\|\|u^*\|$ for all $v\in Y$ and $u^*\in X^*$.
\end{proof}

Using the above proposition, we prove

\begin{thm}\label{FHC}
Let $R\in \mathcal{L}(Y)$ and $T \in \mathcal{L}(X)$ for separable
Banach spaces $X$ and $Y$ and $q\in \mathbb{N}$. Let $\mathcal{D}$
be a norm-dense set in $Y$ and $\Phi$ be a countable weak*-dense set
in $X^*$. Suppose that for each $(y,x^*)\in \mathcal{D}\times \Phi$, there exist sequences $(y_n)_{n\geq 0}$ in $Y$ and
$(x^*_n)_{n\geq 0}$ in $X^*$ such that
\begin{itemize}
\item[\emph{(a)}] the series $\displaystyle \sum_{n=0}^{r} R^{r^q-(r-n)^q}(y)\otimes (T^*)^{r^q-(r-n)^q}(x^*)$ and
$\displaystyle \sum_{n=1}^{\infty} y_{(n+r)^q-r^q}\otimes x^*_{(n+r)^q-r^q}$
are unconditionally convergent in $(\mathcal{L}(X,Y),\|.\|_{op})$, uniformly in $r\geq 0$; and
\item[\emph{(b)}] $R^{n^q}y_{n^q}=y$, $(T^*)^{n^q}x^*_{n^q}=x^*$,
$R^{n^q}y_{m^q}=y_{m^q-n^q}$, and $(T^*)^{n^q}x^*_{m^q}=x^*_{m^q-n^q}$ for all $m>n\geq 0$.
\end{itemize}
Then the following assertions hold.\\ $\emph{(i)}$ If
$T^*(\Phi)\subseteq \Phi$, $\{x_n^*\}\subseteq \Phi$, then
$C_{R,T}$ is $q$-frequently hypercyclic on $(\mathcal{L}(X,Y),\emph{SOT})$.\\
$\emph{(ii)}$ If $Y$ has the AP and the set $\Phi$ is
norm-dense in $X^*$, then $C_{R,T}$ is $q$-frequently hypercyclic on
$\left(\mathcal{K}(X,Y),\|.\|_{op}\right)$ and
$\left(\mathcal{L}(X,Y),\emph{COT}\right)$.
\end{thm}

\begin{proof}
(i) As $Y$ is a separable Banach space, we may assume that $D$ is countable.
Let
\begin{equation*}
\displaystyle
\mathcal{L}_\Phi=\overline{\text{span}}^{\|.\|_{op}}\left \{y\otimes
x^*: y\in Y, x^*\in \Phi\right \},
\end{equation*}
where the closure is taken in the operator norm $\|.\|_{op}$. Then
$\mathcal{L}_\Phi$ is a separable Banach space since the set
\begin{center}
$\displaystyle \mathcal{G}_{D,\Phi}=\left \{ \sum_{n=1}^
{N}y_n\otimes x_n^*: y_n\in D, x_n^*\in \Phi, N\in
\mathbb{N}\right\}$
\end{center}
is countable and norm-dense in $\mathcal{L}_\Phi$.

If $G=\sum_{j\leq N} y_j\otimes x_k^*$, then $C_{R,T}(G)=\sum_{j\leq N} R(y_j)\otimes
T^*(x^*_j)$. As the map $C_{R,T}$ is continuous and
$T^*(\Phi)\subseteq \Phi$, it follows that $C_{R,T}$ takes
$\mathcal{L}_\Phi$ to itself.

To establish the $q$-frequent hypercyclicity of the operator
$C_{R,T}$ in SOT, we first show that $C_{R,T}$ satisfies the
conditions of Theorem \ref{1.3} in the space
$\mathcal{L}_\Phi$. Let $F=\sum_{j=1} ^{k} y_j\otimes x_j^*\in \mathcal{G}_{\mathcal{D},\Phi}$. For each $(y_j,x^*_j)$, let $(y_{j,n})$ and $(x^*_{j,n})$ be some sequences, respectively in $Y$ and $\Phi$, as in the hypothesis. Then
\begin{center}
$\displaystyle F_{n}=\sum_{j\leq k}y_{j,n}\otimes x^*_{j,n} \in \mathcal{L}_\Phi$, $n\geq 0$.
\end{center}
Therefore, by the assumption (a) of our theorem, both the series\\
$\sum_{n\leq r}(C_{R,T})^{r^q-(r-n)^q}(F)=\sum_{j\leq k} \sum_{n\leq r}R^{r^q-(r-n)^q}(y_j)\otimes (T^*)^{r^q-(r-n)^q}(x^*_j)$ and $\sum_{n\geq 0}F_{(n+r)^q-r^q}=\sum_{j\leq k}\sum_{n\geq 0}y_{j,(n+r)^q-r^q}\otimes x^*_{j,(n+r)^q-r^q} $ converge
unconditionally in $\mathcal{L}_\Phi$ with respect to the operator
norm, uniformly in $r\geq 0$.

Further, we have
\begin{align*}
\displaystyle (C_{R,T})^{n^q}F_{m^q}&=\sum_{j=1} ^{k}
R^{n^q}y_{j,m^q}\otimes
(T^*)^{n^q}x^*_{j,m^q}\\&=\begin{cases}
\sum_{j=1} ^{k} y_{j,m^q-n^q}\otimes x^*_{j,m^q-n^q},~~~~ m>n\\
~~F,~~~~~~~~m=n.
\end{cases}
\end{align*}
by hypotheses. Thus $(C_{R,T})^{n^q}F_{m^q}=F_{m^q-n^q}$ if $m>n$ and
$(C_{R,T})^{n^q}F_{n^q}=F, n\geq0$. So $C_{R,T}$ is $q$-frequently hypercyclic on $\mathcal{L}_\Phi$ with respect to the
operator norm topology. As $\mathcal{G}_{D,\Phi}$ is SOT-dense in $\mathcal{L}(X,Y)$ by
Proposition \ref{P}, the operator $C_{R,T}$ is $q$-frequently
hypercyclic on $(\mathcal{L}(X,Y),SOT)$. This establishes part (i).

Let us now prove part (ii). If $Y$ has the AP and the set $\Phi$ is
norm-dense in $X^*$, then by Proposition \ref{P}(3),
$\mathcal{G}_{D,\Phi}$ is dense in $\mathcal{K}(X,Y)$ with respect
to the operator norm and so $\mathcal{L}_\Phi=\mathcal{K}(X,Y)$. Consequently, $C_{R,T}$ is $q$-frequently hypercyclic
on $(\mathcal{K}(X,Y),\|.\|_{op})$ and $(\mathcal{L}(X,Y),\text{COT})$.
\end{proof}

For applications of Theorem \ref{FHC}, we require the following lemmas.

\begin{lem}\label{lomma1}
Let $X$ and $Y$ be Banach spaces. If $\displaystyle
\sum_{n=1}^{\infty} u_{n,j}$ is unconditionally convergent in $Y$, uniformly in $j\geq 0$, and $\{u^*_{n,k}\}\subset X^*$ is such that
$\{u^*_{n,k}: n\geq N_0, k\geq 1\}$ is norm-bounded for some $N_0\in \mathbb{N}$,
then $\sum_{n} u_{n,j}\otimes u^*_{n,k}$ is unconditionally convergent in
$(\mathcal{L}(X,Y),\|.\|_{op})$, uniformly in $j,k\in \mathbb{N}$.
\end{lem}

\begin{proof}
By definition, for a given $\epsilon>0$, one can choose a natural number $N>N_0$ such that
$\|\sum_{n\in F} u_{n,j} \| < \epsilon$ for every finite set set $F\subset [N,\infty)\cap \mathbb{N}$ and all $j\geq 1$.

Let $M\in \mathbb{R}$ be such that $\|u^*_{n,k}\|\leq M$ $\forall$ $n\geq N_0,k\geq 1$. Using the inequality \eqref{qq}, it follows that
\begin{center}
$\displaystyle \left \|\sum_{n\in F} u^*_{n,k}(x) u_{n,j} \right \| \leq 4M~\sup_{G\subseteq F} \left \|\sum_{n\in G} u_{n,j} \right \|<4M\epsilon$,
\end{center}
for $j,k\geq 1$ and $\|x\|\leq 1$. Thus
\begin{center}
$\displaystyle \left \|\sum_{n\in F} u_{n,j}\otimes u^*_{n,k} \right \|_{op}<4M\epsilon$.
\end{center}
\end{proof}

\begin{lem}\label{lomma2}
Let $X$ and $Y$ be Banach spaces. If $\displaystyle
\sum_{n=1}^{j} u_{n,j}$ is unconditionally convergent in $Y$, uniformly in $j\in \mathbb{N}$, and
$\{u^*_{n,j}: n\geq N_0, k\geq 1\}$ is norm-bounded in $X^*$ for some $N_0\in \mathbb{N}$,
then the series $\sum_{n}^{j} u_{n,j}\otimes u^*_{n,j}$ is unconditionally convergent in
$(\mathcal{L}(X,Y),\|.\|_{op})$, uniformly in $j\in \mathbb{N}$.
\end{lem}

Recalling the $q$-FHC Criterion from Theorem
\ref{1.3}, we prove the $q$-frequent hypercyclicity of the
left multiplication operator $\mathfrak{L}_R(S)=RS$. This strengthens a result of A. Bonilla and K.-G. Grosse-Erdmann \cite{BE} about the SOT-frequent
hypercyclicity of $\mathfrak{L}_R$.

\begin{cor} \label{left}
Let $X$ be a separable Banach space and $R\in \mathcal{L}(X)$
satisfy the $q$-FHC Criterion. Then the following hold.\\
$\emph{(i)}$ $\mathfrak{L}_R$ is $q$-frequently hypercyclic on $(\mathcal{L}(X),\text{SOT})$.\\
$\emph{(ii)}$ If $X^*$ is separable and $X$ has the AP, then the
$\mathfrak{L}_R$ is $q$-frequently hypercyclic on
$(\mathcal{K}(X),\|.\|_{op})$ and $(\mathcal{L}(X),\text{COT})$.
\end{cor}
\begin{proof}
In Theorem \ref{FHC}, let $X=Y$ and $T$ be the identity operator on
$X$. Since $X$ is separable, the dual $X^*$ is weak$^*$-separable. So, we can choose $\Phi$ to be any
countable weak$^*$-dense subset of $X^*$. Since $R$ satisfies
the $q$-FHC Criterion, we find a set $D$ satisfying the conditions (a) and (b) of
Theorem \ref{1.3}. That is, for $x\in D$, there exists $(x_n)$ in $X$ with $x_0=x$ such that
\begin{equation*}
\sum_{n\leq r} R^{r^q-(r-n)^q}(x)~and~\sum_{n\geq 1}x_{(n+r)^q-r^q}~are~unconditionally
~~convergent~uniformly~in~r\geq 0,
\end{equation*}
and
\begin{equation*}
R^{n^q}x_{n^q}=x,~ R^{n^q}x_{m^q}=x_{m^q-n^q},~~ m>n,~n\geq 0.
\end{equation*}
Now we verify the condition $(a)$ of Theorem
\ref{FHC}. For each $x^*\in \Phi$, let $x_n^*=x$, $n\geq 0$. By Lemmas \ref{lomma1} and \ref{lomma2} we get (i).

To see (ii), take $\Phi$ as any norm-dense subset of $X^*$ in
Theorem \ref{FHC} and apply Theorem \ref{FHC}(ii).
\end{proof}

Similarly, one can prove the following results. We observe that if $\sum_{n} x_{n,j}$ is unconditionally convergent, uniformly in $j$, then there exists $N\in \mathbb{N}$ such that the set
$\{x_{n,j}:n\geq N,j\geq 1\}$ is bounded.

\begin{cor} \label{right} Suppose $X$ is a separable Banach space and $T\in
\mathcal{L}(X)$. Then the following are true.\\
\emph{(1)} Let $\Phi$ be a countable weak$^*$-dense subset of $X^*$. Suppose that for each $x^*\in \Phi$,
there exists $(x_n^*)$ in $\Phi$ with properties that $x^*_0-x^*$, the series $\sum_{n\leq r}(T^*)^{r^q-(r-n)^q}(x^*)$ and $\sum_{n\geq 1} x^*_{(n+r)^q-r^q}$ are unconditionally convergent in $(X^*,\|.\|)$, uniformly in $r$; and
$(T^*)^{n^q}x^*_{n^q}=x^*$, $(T^*)^{n^q}x^*_{m^q}=x^*_{m^q-n^q}$, $m>n$. If $T^*(\Phi)\subset \Phi$, then $\mathfrak{R}_T$ is $q$-frequently
hypercyclic on $(\mathcal{L}(X),\text{SOT})$. \\
$\emph{(2)}$ Assume that $X^*$ is separable and $X$ has the AP. If
$T^*$ satisfies the $q$-FHC Criterion, then
$\mathfrak{R}_T$ is $q$-frequently hypercyclic on
$(\mathcal{K}(X),\|.\|_{op})$ and $(\mathcal{L}(X),\text{COT})$.
\end{cor}

\begin{prop}\label{corollaryiso}
Let $R,T\in \mathcal{L}(X)$, where $R$ satisfies the $q$-FHC Criterion, and let $\Phi$ be a countable weak$^*$-dense
set in $X^*$. If for each $f\in \Phi$, there exists a bounded $(f_n)$ in $X^*$ such that $f_0=0$ and the set $\{(T^*)^{n}(f):n\geq0\}$ is bounded; and $(T^*)^{n^q}f_{n^q}=f$, $(T^*)^{n^q}f_{m^q}=f_{m^q-n^q}$ for $m>n\geq 0$, then the following hold.\\ $(1)$ If $T^*(\Phi)\subseteq \Phi$, then $C_{R,T}$ is
$q$-frequently hypercyclic on $(\mathcal{L}(X),\text{SOT})$.\\
$(2)$ If $\Phi$ is norm-dense, and $X$ has the AP, then $C_{R,T}$ is
$q$-frequently hypercyclic on $(\mathcal{K}(X),\|.\|_{op})$ and
$(\mathcal{L}(X),\text{COT})$.
\end{prop}

Let us now establish the $q$-frequent hypercyclicity of $C_{R,U}$ for a unitary $U$ and of the conjugate
operator $C_R(S)=RSR^*$.

\begin{cor}\label{conjugate}
Let $R$ satisfy the $q$-FHC Criterion in a
separable Hilbert space $H$. Then the operators $C_R$ and $C_{R,U}$
are $q$-frequently hypercyclic on $(\mathcal{K}(H),\|.\|_{op})$ and
$(\mathcal{L}(H),\text{COT})$, where $U\in \mathcal{L}(H)$ is a
unitary operator.
\end{cor}
\begin{proof}
Since $H$ is a separable Hilbert space, it has the AP. As in the proof of the above results, one can find
a dense set $D$ of $H$ satisfying the conditions in Theorem
\ref{1.3}. Now, in Proposition \ref{corollaryiso} take
$T=R^*$ and $\Phi=D$. The result follows.
\end{proof}

We now proceed to some concrete applications of Theorem \ref{FHC}.
We provide sufficient conditions on the weights $(w_n)$ and
$(\mu_n)$ for the map $C_{B_w,F_\mu}$ to be $q$-frequently
hypercyclic on different Banach algebras of operators on $\ell^p$,
$1\leq p<\infty$, where the backward shift $B_w$ and the forward
shift $F_\mu$ are respectively given by $B_w(e_0)=0$,
$B_w(e_n)=w_ne_{n-1}$, $n\geq 1$ and $F_\mu (e_n)=\mu_{n+1}
e_{n+1}$, $n\geq 0$. Here $\{e_n\}_{n\geq 0}$ is the standard basis
in $\ell^p$.

\begin{prop}\label{unifhc}
If $\displaystyle \lim_{n\rightarrow
\infty}|w_1w_2..w_{(n+r)^q-r^q+i}\mu_1\mu_2..\mu_{(n+r)^q-r^q+j}|=\infty$, uniformly in $r\geq 0$, for all
$i,j\in \mathbb{N}_0$, then $C_{B_w,F_\mu}$ is $q$-frequently
hypercyclic on $(\mathcal{L}(\ell^1),\emph{SOT})$,
$(\mathcal{K}(\ell^p),\|.\|_{op})$ and
$(\mathcal{L}(\ell^p),\emph{COT})$, where $1<p<\infty$.
\end{prop}

\begin{proof}
In Theorem \ref{FHC}, let $X=Y=\ell^p$. To prove the result for
$(\mathcal{L}(\ell^1),\text{SOT})$, let us write $\Phi_0$ for the
linear span of $\{e^*_n:n\geq0\}$ over rationals in $\ell^\infty$
and $\mathcal{D}$ for $\text{span}\{e_n:n\geq0\}$ in $\ell^1$.
Consider the maps $S_w$ and $J_\mu$ given by $S_w(e_n)=\frac {1} {w_{n+1}}e_{n+1}$ and
$J_\mu(e^*_n)=\frac {1} {\mu_{n+1}} e^*_{n+1}$, $n\geq0$. Note that $B_wS_w$ and $F_\mu^*J_\mu$ are identity operators, and
\begin{equation}\label{plus}
\displaystyle S_w^m(e_n)=\frac {1}
{w_{n+1}w_{n+2}...w_{n+m}}e_{n+m}~~\text{and}~~\displaystyle
J_\mu^m(e^*_n)=\frac {1} {\mu_{n+1}\mu_{n+2}...\mu_{n+m}} e^*_{n+m}.
\end{equation}
Put
\begin{center}
$\displaystyle \Phi= \bigcup_{j\geq 0}\Phi_j$, where $\displaystyle
\Phi_{j+1}=\bigcup_{n,k\geq 0}(F_\mu^*)^nJ_\mu^k(\Phi_j)$, $j\geq
0$.
\end{center}
Then the set $\Phi$ is weak$^*$-dense in $\ell^\infty$ as
$\Phi_0\subseteq \Phi$. Clearly $\Phi$ is countable. Further
$F_\mu^*(\Phi)\subseteq \Phi$ and $J_\mu(\Phi)\subseteq \Phi$.

We only consider the series
\begin{center}
$\displaystyle \sum_{n=1}^{r} B^{r^q-(r-n)^q}_w(e_i)\otimes
{(F^*_\mu)}^{r^q-(r-n)^q}(e_j^*)$ and $\displaystyle\sum_{n=1}^{\infty}
S^{(n+r)^q-r^q}_w(e_i)\otimes J^{(n+r)^q-r^q}_\mu(e_j^*)$
\end{center}
for $i,j\geq 0$. As $B_w^n(e_i)=0$ for sufficiently large $n$ and $r^q-(r-n)^q\geq n$, the
first series converges unconditionally in the operator norm, uniformly in $r\geq 0$. For the latter series, it suffices to prove that $\displaystyle \sum_{n\geq 1} a_{n,r} e_{(n+r)^q-r^q+i}\otimes e^*_{(n+r)^q-r^q+j}$ converges unconditionally, uniformly in the operator norm
if $\lim_{n\rightarrow \infty} |a_{n,r}|=0$, uniformly in $r$. But this is immediate as, for $x=(x_n)\in \ell^p$, $1\leq p<\infty$, we have
\begin{center}
$\displaystyle \|\sum_{n\in F} a_{n,r} e_{(n+r)^q-r^q+i}\otimes
e_{(n+r)^q-r^q+j}^*\|=\left \|\sum_{n\in F} a_{n,r}
x_{(n+r)^q-r^q+j} e_{(n+r)^q-r^q+i} \right\|= \big (\sum_{n\in F} |a_{n,r}
x_{(n+r)^q-r^q+j}|^p \big)^{1/p}\leq \max_{n\in F}|a_{n,r}| \|x\|$.
\end{center}
Therefore $C_{B_w,F_\mu}$ is $q$-frequently hypercyclic on
$(\mathcal{L}(\ell^1),SOT)$ by Theorem \ref{FHC}(i).

Since $\ell^p$ has the AP and the set $\Phi$ constructed above is
norm-dense in $\ell^{p^\prime}$, where $1<p<\infty$,
$1/p+1/p^\prime=1$ and $\ell^{p^\prime}$ is the dual of $\ell^p$,
the operator $C_{B_w,F_\mu}$ is $q$-frequently hypercyclic on
$(\mathcal{K}(\ell^p),\|.\|_{op})$ and $(\mathcal{L}(\ell^p),COT)$
by Theorem \ref{FHC}(ii).
\end{proof}

\begin{xrem}
It is evident from the proof of the above proposition that the
result holds for any Banach sequence space $E$ with the AP such that
span $\{e_n:n\geq0\}$, span $\{e^*_n\}$ over rationals are
norm-dense in $E$, $E^*$ respectively and $\displaystyle
\sum_n(w_1w_2..w_{i+(n+r)^q-r^q}\mu_1\mu_2..\mu_{j+(n+r)^q-r^q})^{-1}e_{i+(n+r)^q-r^q}\otimes
e^*_{j+(n+r)^q-r^q}$ converges unconditionally in the operator norm, uniformly in $r$ for all
$i,j\geq 0$.
\end{xrem}




Our next aim is to obtain the bilateral version of Proposition
\ref{unifhc}. For $a=(a_n)_{n\in\mathbb{Z}}$ of a bounded sequence
of nonzero scalars, we define the bilateral backward shift $T_a$ on
the sequence space $\ell^p(\mathbb{Z})$, $1\leq p<\infty$, as
$T_a(e_n)=a_ne_{n-1}$ and the forward shift $S_a$ as
$S_a(e_n)=a_ne_{n+1}$, $n\in \mathbb{Z}$, where $\{e_n\}_{n\in
\mathbb{Z}}$ is the standard basis in $\ell^p(\mathbb{Z})$. Then we
have

\begin{prop}\label{bifhc}
Suppose that, for all $i,j\in \mathbb{Z}$, $\displaystyle \lim_{n\rightarrow
\infty}|a_1a_2..a_{(n+r)^q-r^q+i}b_1b_2..b_{(n+r)^q-r^q+j}|=\infty~~ and~~ \lim_{n\rightarrow \infty}
|a_ia_{i-1}..a_{i-r^q+(r-n)^q+1}b_jb_{j-1}..b_{j-r^q+(r-n)^q+1}|=0$, uniformly in ~$r\in \mathbb{N}_0$.
Then $C_{T_a,S_b}$ is $q$-frequently hypercyclic
on $(\mathcal{L}(\ell^1(\mathbb{Z})),SOT)$, $(\mathcal{K}(\ell^p(\mathbb{Z}),\|.\|_{op}))$ and
$(\mathcal{L}(\ell^p(\mathbb{Z})),COT)$, $1<p<\infty$.
\end{prop}

\begin{proof}
We apply Theorem \ref{FHC}. Choose
$X=Y=\ell^p(\mathbb{Z})$, $\mathcal{D}$=
span$\{e_n:n\in\mathbb{Z}\}$ and $\Phi_0=$
span$\{e^*_n:n\in\mathbb{Z}\}$ over rationals. Define the maps $S$
and $J$ as $S(e_n)=\frac {1} {a_{n+1}} e_{n+1}$ and $J(e_n^*)=\frac {1}
{b_{n+1}} e^*_{n+1}$ for $n\in \mathbb{Z}$. Let
\begin{center}
$\displaystyle \Phi= \bigcup_{j\geq 0}\Phi_j$, where $\displaystyle
\Phi_{j+1}=\bigcup_{n,k\geq 0}{S_b^*}^nJ^k(\Phi_j)$, $j\geq 0$.
\end{center}
The set $\mathcal{D}$ is norm-dense in $\ell^p(\mathbb{Z})$ for $p\in
[1,\infty)$, $\Phi$ is weak$^*$-dense in $\ell^\infty(\mathbb{Z})$ and norm-dense in
$\ell^{p^\prime}(\mathbb{Z})$, where $1/p+1/{p^\prime}=1$ and $1<
p<\infty$. Moreover $T_aS$ is the identity operator on $\mathcal{D}$
and $S_b^*J$ is the identity on $\Phi$. As in Proposition
\ref{unifhc}, $J(\Phi)\subseteq \Phi$ and $S_b^*(\Phi)\subseteq
\Phi$. Further for $n\geq 1$ and $i,j\in \mathbb{Z}$, we have
\begin{center}
~$T_a^n(e_i)=a_ia_{i-1}..a_{i-n+1}e_{i-n}$,~
$(S_b^*)^n(e_j^*)=b_jb_{j-1}..b_{j-n+1}e^*_{j-n}$,
\end{center}
\begin{center}
$S^n(e_i)=\frac {1} {a_{i+1}a_{i+2}..a_{i+n}} e_{i+n}$~ and~
$J^n(e_j^*)=\frac {1} {b_{j+1}b_{j+2}..b_{j+n}}e_{j+n}^*$.
\end{center}
Let $S_n=S^n$ and $J_n=J^n$. Proceeding as in the proof of
Proposition \ref{unifhc}, one can show that the series $\sum_{n\leq
r} T_a^{r^q-(r-n)^q}(e_i)\otimes (S_b^*)^{r^q-(r-n)^q}(e_j^*)$ and $\sum_{n\geq 1}
S_{(n+r)^q-r^q}(e_i)\otimes J_{(n+r)^q-r^q}(e_j^*)$ are unconditionally convergent
in the operator norm, uniformly in $r\geq 0$, by the hypothesis. This proves part (1).

Also since $\ell^p(\mathbb{Z})$ has the AP, and the above set $\Phi$
is norm-dense in $\ell^{p^\prime}(\mathbb{Z})$ ($1<p<\infty$ and
$\frac {1} {p}+\frac {1} {p^\prime}=1$), Theorem \ref{FHC}(ii)
yields (2).
\end{proof}


For the next result, let $\mathbb{C}^N$ be considered as a vector
space over $\mathbb{C}$, and for
$\lambda=(\lambda_1,...,\lambda_N)\in \mathbb{C}^N$, let $D_\mu$ be
the diagonal operator $D_\lambda(f_j)=\lambda_j f_j$ on
$\mathbb{C}^N$ with respect to the standard basis
$\{f_1,f_2,..,f_N\}$ of $\mathbb{C}^N$, where $N\in\mathbb{N}$. Then
we have

\begin{prop}Suppose that $|\lambda_j|\geq 1$ for each $1\leq j\leq N$.\\
$(1)$ If $\displaystyle \sum_{n\geq 1} \frac{1}
{|\mu_1\mu_2...\mu_{(n+r)^q-r^q+i}|}<\infty$ uniformly in $r$, for each $i\geq 0$, then
$C_{D_\lambda,F_\mu}$ is $q$-frequently
hypercyclic on $(\mathcal{L}(\ell^1,\mathbb{C}^N),\text{SOT})$.\\
$(2)$ If $1<p<\infty$ and $\displaystyle \sum_{n\geq 1} \frac {1}
{|\mu_1\mu_2...\mu_{(n+r)^q-r^q+i}|^p}<\infty$ uniformly in $r$, for all $i\geq 0$, then
$C_{D_\lambda,F_\mu}$ is $q$-frequently hypercyclic on
$(\mathcal{K}(\ell^p,\mathbb{C}^N),\|.\|_{op})$ and
$(\mathcal{L}(\ell^p,\mathbb{C}^N),COT)$.
\end{prop}
\begin{proof}
In Theorem \ref{FHC}, take $X=\ell^p$ and $Y=\mathbb{C}^N$. Consider
the set $\Phi$ and the maps $J_{n}$ as in the proof of Proposition
\ref{unifhc}. Choose $D=\mathbb{C}^N$ and $S_{n}=S_\lambda^{n}$,
where $S_\lambda (f_j)= \frac {1} {\lambda_j} f_j$. For a fixed
$i\geq 0$, the series $\sum_{n\leq r} D_\lambda^{r^q-(r-n)^q} (f_j)\otimes (F_\mu^*)^{r^q-(r-n)^q}(e_i)$ is clearly
unconditionally convergent in the operator norm, uniformly in $r$. Since
$S_\lambda ^n(f_j)= {\lambda_j^{-n}} f_j$ and $|\lambda _j|\geq 1$,
the set $\{S_{n}(f_j):n=0,1,2,...\}$ becomes bounded.
Consequently, the series $\sum_{n\geq 1} S_{(n+r)^q-r^q}(f_j) \otimes J_{(n+r)^q-r^q}(e_i)$
converges unconditionally, uniformly in $r$ by Lemma \ref{lomma1}.
\end{proof}

So far, we have considered applications of Theorem \ref{FHC} to
maps on Banach algebras of operators. Now we turn to the $q$-frequent
hypercyclicity of the conjugate operator $C_R(S)=RSR^*$ defined on the real
subspace $\mathcal{S}(H)$ of $\mathcal{L}(H)$, consisting of all
self-adjoint operators on a separable Hilbert space $H$. H. Petersson \cite{H} showed that if $R$ satisfies the Hypercyclicity Criterion, then $C_R$ is hypercyclic on the norm-closure of span $\{h\otimes h:h\in H\}$ over $\mathbb{R}$, and hence COT-hypercyclic on $\mathcal{S}(H)$. A standard application of
the $q$-FHC criterion, using Lemmas \ref{lomma1} and \ref{lomma2} yield the following:
\begin{prop} \label{FHCselfadjoint}
If $R\in \mathcal{L}(H)$ satisfies the $q$-FHC Criterion, then $C_R$ is $q$-frequently hypercyclic on
$(\mathcal{S}(H),\emph{COT})$.
\end{prop}

\section{$q$-Frequent Hypercyclicity in $S_p(H)$}

In this section we provide a sufficient condition for $C_{R,T}$ to be $q$-frequently hypercyclic on $S_p(H)$ for a separable Hilbert space $H$. Let us begin
with the $q$-frequent hypercyclicity of the left multiplication operator $\mathfrak{L}_R$ on $S_p(H)$, $1\leq p<\infty$, which is an easy application of the $q$-FHC Criterion.

\begin{prop} \label{left-schatten}
Suppose that $R$ is an operator satisfying the $q$-Frequent
Hypercyclicity Criterion in a Banach space $X$ with separable dual.\\
$(1)$ If $(\mathcal{I}(X),\|.\|_\mathcal{I})$ is a separable Banach
ideal in $\mathcal{L}(X)$ such that the finite rank operators on $X$
are $\|.\|_\mathcal{I}$-dense in $\mathcal{I}(X)$, and $\|x\otimes
x^*\|_\mathcal{I}=\|x\|\|x^*\|$ for all $x\in X$ and $x^*\in X^*$,
then the left multiplication operator $\mathfrak{L}_R$ is
$q$-frequently
hypercyclic on $\mathcal{I}(X)$.\\
$(2)$ If $X$ is a separable Hilbert space, then $\mathfrak{L}_R$ is
frequently hypercyclic on $S_p(X)$, $1\leq p<\infty$.
\end{prop}

\begin{proof}
The proof of (1) is omitted. To see that (2) is true, recall that
$\|u\otimes v\|_p=\|u\| \|v\|$ for all $u,v \in X$ and the finite rank operators on $X$ form a
dense subspace of $S_p(X)$.
\end{proof}

For the main theorem of this section, we state the following result on summability of a series in $S_p(H)$, cf. \cite{AC}, p. 152.

\begin{lem}\label{Lemma1}
Let $\displaystyle \{T_n\}_{n=1}^{\infty}\subset \mathcal{L}(H)$ be
such that $T_n^*T_m=T_nT_m^*=0$ whenever $m\neq n$. Then for $1\leq
p <\infty$
\begin{center}
$\|\sum_{n} T_n\|_p^p= \sum_{n} \|T_n\|_p^p$.
\end{center}
\end{lem}

We prove:

\begin{thm}\label{FHC-Schatten}
Let $1\leq p<\infty$, $R,T\in \mathcal{L}(H)$ and $D_1,D_2 \subset
H$. Let $D_1$ and $D_2$ both span dense subspaces of $H$. If for each $(x,y)\in
D_1\times D_2$, there exist sequences $(x_n,y_n)\in H\times H$ with $(x_0,y_0)=(x,y)$ and
\begin{itemize}
\item[\emph{(a)}] $\displaystyle \sum_{n=1}^{r} \|R^{r^q-(r-n)^q}(x)\|^p\|(T^*)^{r^q-(r-n)^q}(y)\|^p<\infty$ and $\displaystyle
\sum_{n=1}^{\infty} \|x_{(n+r)^q-r^q}\|^p\|y_{(n+r)^q-r^q}\|^p<\infty$ uniformly in $r\geq 0$,
\item[\emph{(b)}] $\big<R^{n}(x),R^{m}(x)\big>=\big<S_{n}(x),S_{m}(x)\big>=\big<(T^*)^{n}(y),(T^*)^{m}(y)\big>\\
=\big<J_{n}(y),J_{m}(y)\big>=0,$ for $m\neq n$; and
\item[\emph{(c)}]$R^{n^q}x_{n^q}=x$, $(T^*)^{n^q}y_{n^q}=y$, $R^{n^q}x_{m^q}=x_{m^q-n^q}$,
$(T^*)^{n^q}y_{m^q}=y_{m^q-n^q}$, $\forall m>n\geq 0$,
\end{itemize}
then $C_{R,T}$ is $q$-frequently hypercyclic on $(S_p(H),\|.\|_p)$.
\end{thm}
\begin{proof}
Let $\Delta=\text{span}\{x\otimes y: x\in D_1, y\in D_2\}$. Note
that $\Delta$ can also be written as the span of the set $\{x\otimes
y: x\in \text{span}D_1, y\in \text{span} D_2\}$. Since span$D_1$ and
span$D_2$ are dense in $H$, it can be proved that $\Delta$ is dense
in $S_p(H)$, $1\leq p <\infty$.

Let $F=\sum_{k=1}^{N} a_k x_k\otimes y_k \in \Delta$. Corresponding to $x_k$ and $y_k$, we obtain sequences $(x_{k,n})$ and $(y_{k,n})$ as in the hypothesis.
Set $F_n=\sum_{k=1}^{N} a_k x_{k,n}\otimes y_{k,n}$. Consider the series $\sum_{n\leq r} C_{R,T}^{r^q-(r-n)^q}(F)=\sum_{k=1}^{N} a_k
\sum_{n=1}^{\infty} R^{r^q-(r-n)^q}(x_k)\otimes (T^*)^{r^q-(r-n)^q}(y_k)$ and
$\sum_n F_{(n+r)^q-r^q}=\sum_{k=1}^{N}a_k \sum_{n=1}^{\infty}
x_{k,(n+r)^q-r^q}\otimes y_{k,(n+r)^q-r^q}$. It suffices to prove that $\sum_{n\leq r}R^{r^q-(r-n)^q}(x)\otimes
(T^*)^{r^q-(r-n)^q}(y)$ and $\sum_n x_{k,(n+r)^q-r^q}\otimes y_{k,(n+r)^q-r^q}$ are unconditionally convergent in
$S_p(H)$, uniformly in $r$, for all $x\in D_1$ and $y\in D_2$.

Write $T_{n,r}=R^{r^q-(r-n)^q}(x)\otimes (T^*)^{r^q-(r-n)^q}(y)$, $n\geq 1$. Then
$T_{n,r}^*=(T^*)^{r^q-(r-n)^q}(y)\otimes R^{r^q-(r-n)^q}(x)$. If $<.>$ is the inner
product in $H$, then $T_{n,r}^*(z)=\big<z,R^{r^q-(r-n)^q}(x)\big>(T^*)^{r^q-(r-n)^q}(y)$
and $T_{m,r}(z)=\big<z,(T^*)^{r^q-(r-m)^q}(y)\big>R^{r^q-(r-m)^q}(x)$ and so
\begin{center}
$T_{n,r}^*T_{m,r}(z)=\big<z,(T^*)^{r^q-(r-m)^q}(y)\big>\big<R^{r^q-(r-m)^q}x,R^{r^q-(r-n)^q}x\big>(T^*)^{r^q-(r-n)^q}(y).$
\end{center}
Similarly
\begin{center}
$T_{n,r}T_{m,r}^*(z)=\big<z,R^{r^q-(r-m)^q}x\big>\big<(T^*)^{r^q-(r-m)^q}y,(T^*)^{r^q-(r-n)^q}y\big>R^{r^q-(r-n)^q}x$.
\end{center}
From part $(b)$ in the hypotheses, we get $T^*_{n,r}T_{m,r}=T_{n,r}T_{m,r}^*=0$,
$m\neq n$. Since $\|u\otimes v\|_p=\|u\|\|v\|$ for all $u,v\in H$,
Lemma \ref{Lemma1} and the hypothesis $(a)$ yield that
$\sum_{n\leq r}R^{r^q-(r-n)^q}(x)\otimes (T^*)^{r^q-(r-n)^q}(y)$ is unconditionally convergent in
$S_p(H)$, uniformly in $r$. Similarly, one can obtain that $\sum_n
x_{k,(n+r)^q-r^q}\otimes y_{k,(n+r)^q-r^q}$ is unconditionally convergent in $S_p(H)$, uniformly in $r$.
Thus the condition $(a)$ of Theorem \ref{FHC} is
satisfied by $C_{R,T}$ in $S_p(H)$.
\end{proof}

Next, from Theorem \ref{FHC-Schatten}, we obtain conditions on the
weight sequences $(w_n)$ and $(\mu_n)$ that are sufficient for the
$C_{B_w,F_\mu}$ on $S_p(\ell^2)$ and $C_{T_a,S_b}$ on
$S_p(\ell^2(\mathbb{Z}))$ to be $q$-frequently hypercyclic, $1\leq
p<\infty$.

\begin{prop}\label{unifhcS}
$(1)$ If $\displaystyle
\sum_{n=1}^{\infty}|w_1w_2..w_{(n+r)^q-r^q+i}\mu_1\mu_2..\mu_{(n+r)^q-r^q+j}|^{-p}<\infty$ uniformly in $r\geq 0$
for all $i,j\geq0$, then $C_{B_w,F_\mu}$ is $q$-frequently hypercyclic
on $S_p(\ell^2)$.\\
$(2)$ If $\displaystyle \sum_{n=1}^{\infty}|a_1a_2..a_{(n+r)^q-r^q+i}b_1b_2..b_{(n+r)^q-r^q+j}|^{-p}<\infty$ and\\ $\displaystyle
\sum_{n=0}^{r}{|a_ia_{i-1}..a_{i-r^q+(r-n)^q+1}b_jb_{j-1}..b_{j-r^q+(r-n)^q+1}
|^p}<\infty$ uniformly in $r\geq 0$ for all $i,j\in\mathbb{Z}$, then $C_{T_a,S_b}$ is
$q$-frequently hypercyclic on $S_p(\ell^2(\mathbb{Z}))$.
\end{prop}

\begin{proof}
To prove part (1), choose $D_1=D_2=\{e_n: n\geq 0\}$, the standard
orthonormal basis in $\ell^2$. Let $S_n$ and $J_n$ be the maps as
considered in the proof of Proposition \ref{unifhc}, i.e., $S_{n}(e_i)=S_w^{n}(e_i)=\frac {1}
{w_{i+1}..w_{i+n}} e_{i+n}$ and $J_{n}(e_j)=J_\mu^{n}(e_j)=\frac {1}
{\mu_{j+1}..\mu_{j+n}}e_{j+n}$. Note that $\big<S_n(e_j),S_m(e_j)\big>=\big<J_n(e_j),J_m(e_j)\big>=0$ for
$n\neq m$. As $B_w^n(e_i)=0$ for sufficiently large $n$, we have that\\ $\sum_{n\leq r} \|B_w^{r^q-(r-n)^q}(e_i)
\|^p\|(F_\mu^*)^{r^q-(r-n)^q}(e_j)\|^p<\infty$ uniformly in $r$. Moreover, from the
hypothesis, we have $\sum_n \|S_{(n+r)^q-r^q}(e_i)\|^p
\|J_{(n+r)^q-r^q}(e_j)\|^p<\infty$ uniformly in $r$. Now Theorem \ref{FHC-Schatten} yields
the result.

To prove part (2), consider the maps $S_n$ and $J_n$ in the proof of
Proposition \ref{bifhc} and choose $D_1=D_2=\{e_n:n\in
\mathbb{Z}\}$, the standard orthonormal basis in
$\ell^2(\mathbb{Z})$. Then, for $i,j\in \mathbb{Z}$, we have
\begin{equation*}
S_n(e_i)=S^n(e_i)=\frac {1} {a_{i+1}a_{i+2}..a_{i+n}}
e_{i+n}~~\text{and}~~ J_n(e_j)=J^n(e_j)=\frac {1}
{b_{j+1}b_{j+2}..b_{j+n}}e_{j+n},
\end{equation*}
and
\begin{equation*}
T_a^n(e_i)=a_ia_{i-1}...a_{i-n}
e_{i-n-1}~~\text{and}~~(S_b^*)^n(e_j)=b_jb_{j-1}...b_{j-n}
e_{j-n-1}.
\end{equation*}
It follows by the hypotheses that the series $\displaystyle \sum_{n\geq 1} \|S_{(n+r)^q-r^q}(e_i)\|^p
\|J_{(n+r)^q-r^q}(e_j)\|^p$
as well as\\ $\displaystyle\sum_{n\leq r} \|T_a^{r^q-(r-n)^q}(e_i)\|^p
\|(S_b^*)^{r^q-(r-n)^q}(e_j)\|^p $ converges uniformly $r\geq 0$.
\end{proof}

\section{Frequent Hypercyclicity in $S_p(H^2(\mathbb{D}))$ and $\mathcal{N}(\ell^p)$}

This section includes results on frequent hypercyclicity of specific
operators of the form $C_{R,T}$ defined on the $p$th Schatten von-Neumann class of operators on the Hardy space $H^2(\mathbb{D})$,
$1\leq p<\infty$, as well as on the space $\mathcal{N}(\ell^p)$ of
all nuclear operators on $\ell^p$, $1<p<\infty$. Let us recall the \textbf{Eigenvalue Criterion}, due to S. Grivaux.
\begin{prop} \cite{Grivaux} \label{Grivaux}
Let $X$ be a separable, complex Banach space and $T\in \mathcal{L}(X)$. If for every countable subset $D$ of the unit circle $S^1$, the set $\bigcup_{\alpha \in S^1\setminus D} Ker(T-\alpha I)$ spans a dense subspace of $X$, then $T$ is frequently hypercyclic.
\end{prop}

Corresponding to a sequence $\beta=(\beta_n)$,
$\beta_n>0$, $n\geq 0$, let $\big(H^\beta(\mathbb{D}),<.>\big)$ be a
Hilbert space of complex functions, analytic on the open unit disc
$\mathbb{D}$ such that the evaluation mappings $f\rightarrow f(z)$
are continuous at each $z\in \mathbb{D}$, i.e. there exists $k_z\in
H^\beta(\mathbb{D})$ such that $f(z)=\big<f,k_z\big>$ for each $f\in
H^\beta(\mathbb{D})$. Such a function $k_z$ is called \textbf{a reproducing
kernel} at $z\in \mathbb{D}$. Also, assume that
$\{e_n\}_{n=0}^{\infty}$ forms an orthonormal basis for
$H^\beta(\mathbb{D})$, where $e_n(z)=\beta_n z^n$. Note that when
$\beta_n=1$ for all $n\geq 0$, we have the Hardy space
$H^2(\mathbb{D})$.

Let $M_\varphi$ be the multiplication operator $f(z)\rightarrow
\varphi(z) f(z)$ on $H^\beta(\mathbb{D})$, corresponding to an
analytic function $\varphi$ on $\mathbb{D}$ and $M^*_\varphi$ be the
Hilbert space adjoint of $M_\varphi$. Our aim is to establish the
frequent hypercyclicity of $C_{M^*_\varphi,M_\psi}$ on
$S_p(H^\beta(\mathbb{D}))$, where $M_\varphi$ and $M_\psi$ are
bounded multiplication operators on $H^\beta(\mathbb{D})$
corresponding to the analytic functions $\varphi$ and $\psi$ on
$\mathbb{D}$. Let us first prove

\begin{lem}\label{lem}
Let $\varphi$ and $\psi$ be non-zero analytic functions on
$\mathbb{D}$ such that at least one of them is non-constant and
$|\varphi(z)\psi(w)|=1$ for some $z,w \in \mathbb{D}$. Then
\begin{center}
\emph{span} $\{k_z\otimes k_w: \overline{\varphi(z)}\psi(w)\in
S^1\setminus D\}$
\end{center}
is dense in the space $S_1(H^\beta(\mathbb{D}))$ for every countable
set $D\subset S^1$.
\end{lem}

\begin{proof}
We write $\mathcal{H}=H^\beta(\mathbb{D})$ and
$\overline{\varphi(\mathbb{D})}=\{\overline{\varphi(z)}:z\in
\mathbb{D}\}$. By the open mapping theorem for analytic functions,
the set
$\overline{\varphi(\mathbb{D})}\psi(\mathbb{D})=\{\overline{\varphi(z)}\psi(w):
z,w\in \mathbb{D}\}=\bigcup_{z\in
\mathbb{D}}(\overline{\varphi(z)}\psi(\mathbb{D}))$ is non-empty and
open. Hence there exists an open arc $\Gamma$ in $S^1$ such that
$\Gamma \subset \overline {\varphi(\mathbb{D})}\psi(\mathbb{D})$; let us assume that
this arc $\Gamma$ is the maximal one.

Consider the set
\begin{center}
$U\times V=\{(z,w)\in \mathbb{D}\times \mathbb{D}:
\overline{\varphi(z)}\psi(w)\in \Gamma\setminus D\}$.
\end{center}
We claim that $U$ is uncountable, and for each $z\in U$, there
exists an uncountable set $V_1\subseteq V$ such that
\begin{equation} \label{kuthira}
\overline{\varphi(z)}\psi(w) \in \Gamma\setminus D,
~~\text{for}~\text{all}~~ w\in V_1.
\end{equation}
To prove this, assume that both $\varphi$ and $\psi$ are
non-constant. In this case, $\overline{\varphi(\mathbb{D})}$ is a
non-empty open set. If $|\overline{\phi(z)} \psi(w)|=1$ for some
$(z,w)\in \mathbb{D}\times \mathbb{D}$, then
$\psi(w)\overline{\varphi(\mathbb{D})}$ is non-empty and open and
so, we can find an arc $\Gamma_1 \subset \Gamma$ such that
$\Gamma_1\subset \psi(w)\overline{\varphi(\mathbb{D})}$ as above.
Since $D$ is countable and $\Gamma_1 \setminus D \subset
\psi(w)\overline {\varphi(\mathbb{D})}$, the set $U$ has to be
uncountable. Now fix $z\in U$. Then the set
$\overline{\varphi(z)}\psi(\mathbb{D})$, being non-empty and open,
contains $\Gamma_2 \setminus D$ for some sub-arc $\Gamma_2$ of
$\Gamma$. This proves the claim when $\varphi$ and $\psi$ are non-constant.
Now assume that $\psi$ is constant, say $\psi=c$ and $\varphi$ is
non-constant. We can proceed as above to prove that the set $U$ is
uncountable since $c\overline{\varphi(\mathbb{D})}$ is a non-empty
open set containing $\Gamma \setminus D$. Also, since $\psi$ is
constant, we can take $V_1=V=\mathbb{D}$. Finally when $\varphi$ is
constant and $\psi$ is non-constant, we proceed similarly to get the
result. Hence our claim is established.

We now show that $\Lambda=$ span$\{k_z\otimes k_w: z\in U, w\in V\}$
is dense in the space $S_1(H)$. Recall that the trace of $A\in S_1(\mathcal{H})$ is given by
$tr(A)=\sum_{n\geq0}\big<Ae_n,e_n\big>$, where $e_n(z)=\beta_nz^n$, $n\geq 0$. Also we have $S_1(H)^*=\mathcal{L}(H)$
with respect to the duality-pairing $(A,T)=tr(AT)$, $T\in
S_1(H)$ and $A\in \mathcal{L}(H)$.\\ Let $A\in \mathcal{L}(H)$ be such that $tr(AT)=0$ for all
$T\in\Lambda$. For $T=k_z\otimes k_w$, we have
$Te_n=\big<e_n,k_w\big>k_z=e_n(w)k_z=\beta_nw^nk_z$ and
\begin{align*}
tr(AT)&=\sum_{n\geq 0}\big<ATe_n,e_n\big>\\
&=\sum_{n\geq 0}\beta_n w^n \big<k_z,A^*e_n\big>\\
&=\sum_{n\geq 0}\beta_n\overline{(A^*e_n)(z)}w^n.
\end{align*}

\noindent Since $AT\in S_1(H)$, the above power series is well-defined for all
$z,w\in \mathbb{D}$. Hence it is an analytic function in the
variable $w$ for a fixed $z\in \mathbb{D}$. For $z\in U$, there exists an uncountable set $V_1$ such that
\eqref{kuthira} holds. Since $V_1$ is uncountable, it has a limit point in
$\mathbb{D}$. As $\beta_n>0$ for all $n\geq 0$ and $tr(AT)=0$, it
follows that the coefficients of
the above power series are all zero, i.e., $A^*(e_n)(z)=0$ for all $n\geq
0$. Similarly, since $z \in U$ is arbitrary and $U$ is uncountable,
we have $A^*(e_n)=0$, $\forall$ $n\geq 0$. As $\{e_n:n\geq 0\}$ spans a dense subspace of $\mathcal{H}$,
we conclude that $A=0$. Therefore the set $\Lambda$ is dense in $S_1(\mathcal{H})$. The
proof is now complete.
\end{proof}

The above lemma yields

\begin{thm}\label{ghardy}
Let $\varphi$ and $\psi$ be non-zero analytic functions on
$\mathbb{D}$ such that the corresponding multiplication operators
are bounded on $H^\beta(\mathbb{D})$. If one of the maps $\varphi$
and $\psi$ is non-constant and $|\varphi(z)\psi(w)|=1$ for some
$z,w\in \mathbb{D}$, then $C_{M^*_\varphi,M_\psi}$ is frequently
hypercyclic on $(S_p(H^\beta(\mathbb{D})),\|.\|_p)$,
$(\mathcal{K}(H^\beta(\mathbb{D})),\|.\|_{op})$, and
$((\mathcal{L}(H^\beta(\mathbb{D})),\text{COT}))$.
\end{thm}

\begin{proof}
For $z,w\in \mathbb{D}$, consider $k_z\otimes k_w\in S_1(H)$. Since $M^*_\varphi(k_z)=\overline{\varphi(z)} k_z$, we have
\begin{center}
$C_{M^*_\varphi,M_\psi}(k_z\otimes
k_w)=\overline{\varphi(z)}\psi(w) (k_z\otimes k_w)$.
\end{center}
Thus $k_z\otimes k_w$ is an eigen vector for
$C_{M^*_\varphi,M_\psi}$ corresponding to the eigen value
$\overline{\varphi(z)}\psi(w)$. Now by Lemma \ref{lem}, span
$\{k_z\otimes k_w:\overline{\varphi(z)}\psi(w)\in S^1\setminus D\}$
is dense in $S_1(H^\beta(\mathbb{D}))$ for any countable set
$D\subset S^1$. Hence $C_{M^*_\varphi,M_\psi}$ is frequently
hypercyclic on $(S_1(H^\beta(\mathbb{D})),\|.\|_1)$ by Proposition \ref{Grivaux}. Since the embeddings
$S_1(H^\beta(\mathbb{D}))\hookrightarrow S_p(H^\beta(\mathbb{D}))
\hookrightarrow \mathcal{K}(H^\beta(\mathbb{D})) \hookrightarrow
(\mathcal{L}(H^\beta(\mathbb{D})),COT)$ are continuous and have
dense range, it follows that $C_{M^*_\varphi,M_\psi}$ is frequently
hypercyclic on each of the these spaces.
\end{proof}

As noted in the beginning of this section, the Hardy space
$H^2(\mathbb{D})$ is a special case of $H^\beta(\mathbb{D})$. In
this case, we have the following characterization for
$C_{M^*_\varphi,M_\psi}$ on $S_p(H^2(\mathbb{D}))$ and
$\mathcal{K}(H^2(\mathbb{D}))$.

\begin{thm}\label{hardy}
Let $\varphi$ and $\psi$ be non-zero, bounded and analytic on
$\mathbb{D}$, with one of them being non-constant. Then
$C_{M^*_\varphi,M_\psi}$ is frequently hypercyclic on
$(S_p(H^2(\mathbb{D})),\|.\|_p)$,
$(\mathcal{K}(H^2(\mathbb{D})),\|.\|_{op})$ and
$(\mathcal{L}(H^2(\mathbb{D})),COT)$ if $|\varphi(z)\psi(w)|=1$ for
some $z,w\in \mathbb{D}$. Conversely, if $C_{M^*_\varphi,M_\psi}$ is
frequently hypercyclic on
$(\mathcal{K}(H^2(\mathbb{D})),\|.\|_{op})$ or
$(S_p(H^2(\mathbb{D})),\|.\|_{p})$, then $|\varphi(z)\psi(w)|=1$ for
some $z,w\in \mathbb{D}$.
\end{thm}
\begin{proof}
We know that $M_\varphi \in
\mathcal{L}(H^2(\mathbb{D}))$ if and only if $\varphi$ is bounded on
$\mathbb{D}$. Thus by the preceding theorem, if $|\varphi(z)\psi(w)|=1$ for some $z,w \in \mathbb{D}$, then
$C_{M^*_\varphi,M_\psi}$ is frequently hypercyclic.

Let $G=\varphi(\mathbb{D})\psi(\mathbb{D})$.
Assume that the converse is not true, i.e., $G\cap S^1= \phi$. Then
$G\subseteq \mathbb{D}$ or $G\subseteq
\mathbb{C}\backslash\overline{\mathbb{D}}$ since $G$ is non-empty
and open, where $\overline{\mathbb{D}}$ is the closed unit disc in
$\mathbb{C}$. In case $G\subseteq \mathbb{D}$, then
$\|\varphi\|_{\infty}\|\psi\|_{\infty}=\sup_{z\in
\mathbb{D}}|\varphi(z)|\sup_{w\in \mathbb{D}}|\psi(w)|=\sup_{z,w \in
\mathbb{D}} |\varphi (z)\psi (w)|\leq 1$, and consequently,
\begin{center}
$\|C_{M_\varphi^*,M_\psi}\| \leq \|M^*_\varphi\|\|M_\psi\| \leq
\|\varphi\|_{\infty}\|\psi\|_{\infty} \leq 1$.
\end{center}
In the latter case when $G\subseteq
\mathbb{C}\setminus\overline{\mathbb{D}}$, we have $\displaystyle
\inf_{z,w\in \mathbb{D}}|\varphi(z)\psi(w)|\geq 1$, and so,
$\displaystyle \inf_{z\in \mathbb{D}} |\varphi(z)|>0$ and
$\displaystyle \inf_{w\in \mathbb{D}}|\psi(w)|>0$. Thus the
functions $\varphi^{-1}$ and $\psi^{-1}$ are bounded and analytic on
$\mathbb{D}$, and the corresponding operators $M_{\varphi^{-1}}$,
$M_{\psi^{-1}}$ are bounded on $H^2(\mathbb{D})$. Moreover
$(M_\varphi)^{-1}=M_{\varphi^{-1}}$ and
$(M^*_\varphi)^{-1}=M^*_{\varphi^{-1}}$. Then we observe that
\begin{center}
$\|{(C_{M_\varphi^*,M_\psi})}^{-1}\|=\|C_{M_{\varphi^{-1}}^*,M_{\psi^{-1}}}\|\leq
\|M^*_{\varphi^{-1}}\|\|M_{\psi^{-1}}\|\leq
\|\varphi^{-1}\|_{\infty}\|\psi^{-1}\|_{\infty}\leq 1$.
\end{center}
\noindent Thus, in both the cases, the operator $C_{M^*_\varphi, M_\psi}$ is not
hypercyclic on $\mathcal{K}(H^2(\mathbb{D}))$. This contradiction proves that $G\cap S^1\neq \phi$.
\end{proof}

As a special case, we state below the frequent hypercyclicity of the conjugate operator
$C_{M^*_\varphi}$ for the Hardy space $H=H^2(\mathbb{D})$.
\begin{prop}
If $\varphi$ is non-constant and $|\varphi(z)|=1$ for some $z\in
\mathbb{D}$, then the conjugate map $C_{M^*_\varphi}$ is frequently
hypercyclic on $(S_p(H),\|.\|_p)$, $(\mathcal{K}(H),\|.\|_{op})$ and
$(\mathcal{L}(H),COT)$. Conversely, if $C_{M^*_\varphi}$ is
frequently hypercyclic on $(\mathcal{K}(H),\|.\|_{op})$, then $\varphi$ is
non-constant and $\varphi(\mathbb{D})\cap S^1\neq \phi$.
\end{prop}

\begin{proof}
Immediate from the preceding theorem.
\end{proof}

The above characterization is not true for multiplication operators
on all Hilbert spaces of analytic functions, e.g. consider

\begin{exa}
Let $\displaystyle \mathcal{H}=\big \{ f(z)=\sum_{n\geq 0} a_n z^n:
\|f\|^2=\sum_{n\geq 0}(n+1)^2|a_n|^2<\infty \big \}$. Then $\mathcal{H}$ is
a reproducing kernel Hilbert space and the multiplication operator $M_\varphi$ corresponding to
$\varphi(z)=z$ acts as the shift
\begin{center}
$\displaystyle M_\varphi(e_n)(z)=ze_n(z)=\frac{1} {n+1} z^{n+1}=
\frac {n+2} {n+1} e_{n+1}(z)$
\end{center}
with respect to the orthonormal basis $e_n(z)=\frac {1} {n+1} z^n$, $n\geq 0$. Now $w_1w_2...w_n=n+1$ implies
that $\sum_n \frac {1} {(w_1w_2..w_n)^2} <\infty$ and consequently,
the adjoint $M^*_\varphi$ satisfies the FHC Criterion on $\mathcal{H}$. Hence by Corollary
\ref{conjugate}, $C_{M^*_\varphi}$ is frequently hypercyclic on the
spaces $(S_p(H),\|.\|_p)$, $(\mathcal{K}(\mathcal{H}),\|.\|_{op})$
and $(\mathcal{L}(\mathcal{H}),COT)$. However there are no $z,w\in
\mathbb{D}$ such that $|\varphi(z)\varphi(w)|=1$.
\end{exa}

In the spirit of Theorem \ref{ghardy}, let us prove a similar result
about $C_{\phi(B),\psi(F)}$ defined on spaces of operators on
$\ell^p$, $1<p<\infty$, where $\varphi(B)$ and $\varphi(F)$ are
functions of the unweighted backward and forward shifts
respectively. If $\varphi(z)=\sum_{n\geq 0} a_n z^n$ is an analytic
function on some neighborhood of the closed disc
$\overline{\mathbb{D}}$, then $\varphi(B)=\sum_n a_n B^n$ and
$\varphi(F)=\sum_n a_n F^n$ are bounded operators on $\ell^p$,
$1\leq p<\infty$. Moreover, the Banach space adjoint of $\varphi(B)$
is $\phi(F)$ and that of $\varphi(F)$ is $\varphi(B)$. Note that if $f_\lambda =(1,\lambda,\lambda^2,...)$ and $|\lambda|<1$, then $\varphi(B) f_\lambda =\varphi(\lambda )f_\lambda$. In \cite{DE}, R. Delaubenfels and H. Emamirad proved that if $\varphi(\mathbb{D})\cap S^1\neq \phi$, then $\varphi(B)$ is hypercyclic on $\ell^p$. We now have the following result, which can be proved using Proposition \ref{Grivaux}.
\begin{prop}
If $\varphi$ is non-constant and $\varphi(\mathbb{D})\cap S^1 \neq
\phi$, then $\varphi(B)$ is frequently hypercyclic on $\ell^p$,
$1\leq p<\infty$.
\end{prop}

Let $\mathcal{N}(\ell^p)$ denote the space of all nuclear operators on $\ell^p$. Then the trace $tr(T)=\sum_n
x_n^*(x_n)$ of $T=\sum_n x_n\otimes x_n^*\in \mathcal{N}(\ell^p)$,
$1<p<\infty$. Then the dual of $\mathcal{N}(\ell^p)$ is identified
with $\mathcal{L}(\ell^p)$ via the trace-duality $(S,T)=tr(TS)$, where $T\in \mathcal{N}(\ell^p)$ and $S\in
\mathcal{L}(\ell^p)$, cf. \cite{BST}, Theorem 16.50. The one-rank operator $f_\lambda \otimes f_\mu$ on $\ell^p$ is given
by $x\rightarrow f_\mu(x) f_\lambda$.

\begin{lem}\label{lemma-nuclear}
Let $\varphi$ and $\psi$ be non-zero functions analytic on some
neighborhoods of the closed disc $\overline{\mathbb{D}}$ with one of
them being non-constant and $|\varphi(\lambda)\psi(\mu)|=1$ for some
$\lambda,\mu \in \mathbb{D}$. Then
\begin{equation*}
\emph{span} \{f_\lambda \otimes f_\mu: \varphi(\lambda)\psi(\mu)\in
S^1\setminus D\}
\end{equation*}
is dense in $\mathcal{N}(\ell^p)$ for every countable set $D\subset
S^1$ and $1<p<\infty$.
\end{lem}

\begin{proof}
Invoking the proof of Lemma \ref{lem}, we can find an arc $\Gamma
\subseteq S^1\cap \varphi(\mathbb{D})\psi(\mathbb{D})$. Write
$U\times V=\{(\lambda,\mu)\in \mathbb{D}\times \mathbb{D}:
\varphi(\lambda)\psi(\mu)\in \Gamma\setminus D\}$. Then, for each
$\mu \in V$, there exists an uncountable set $U_1\subset \mathbb{D}$ such that
for $\lambda\in U_1$ we have $\varphi(\lambda)\psi(\mu) \in \Gamma
\setminus D$. 

Let us now prove that $\Delta=$ span $\{f_\lambda \otimes f_\mu:
\lambda \in U, \mu \in V \}$ is dense in $\mathcal{N}(\ell^p)$. For this, let $S \in \mathcal{L}(\ell^p)$ such that $tr(TS)=0$ for all $T\in \Delta$. In particular, if $T=f_\lambda \otimes f_\mu$ for $\lambda \in U_1$, then $tr(f_\lambda \otimes S^*f_\mu)= (S^*f_\mu)(f_\lambda)=0$. Since $U_1$ has limit points in $\mathbb{D}$, we have that span$\{f_\lambda: \lambda \in U_1\}$ is dense in $\ell^p$. Thus $S^*(f_\mu)=0$ for all $\mu
\in V$. As span$\{f_\lambda: \mu \in V \}$ is dense in $\ell^{p^*}$,
the dual of $\ell^p$, $S=0$.
\end{proof}

Using the above lemma, we prove the frequent hypercyclicity of
$C_{\varphi(B),\psi(F)}$ as follows.

\begin{thm}\label{prop-lp}
Suppose $\phi$ and $\psi$ are non-zero analytic maps on some
neighborhoods of the closed disc $\overline{\mathbb{D}}$ such that
$|\phi(z)\psi(w)|=1$ for some $z,w\in \mathbb{D}$, with one of
$\varphi$ and $\psi$ being non-constant. Then $C_{\phi(B),\psi(F)}$
is frequently hypercyclic on $(\mathcal{N}(\ell^p),\|.\|_{nu})$,
$(\mathcal{K}(\ell^p),\|.\|_{op})$ and $(\mathcal{L}(\ell^p),COT)$
for $1<p<\infty$.
\end{thm}

\begin{proof}
Let $\lambda, \mu \in \mathbb{D}$. Since
$\phi(B)(f_\lambda)=\phi(\lambda)f_\lambda$, we get
\begin{center}
$C_{\phi(B),\psi(F)}(f_\lambda \otimes
f_\mu)=\varphi(B)(f_\lambda)\otimes \psi(B)(f_\mu)=\phi(\lambda)
\psi(\mu) (f_\lambda \otimes f_\mu)$,
\end{center}
where $f_\lambda=(1,\lambda,\lambda^2,...)$. From Lemma
\ref{lemma-nuclear}, it follows that span $\{f_\lambda \otimes f_\mu:
\varphi(\lambda)\psi(\mu)\in S^1\setminus D\}$ is dense in $\mathcal{N}(\ell^p)$ and the proof is complete by Proposition
\ref{Grivaux}.
\end{proof}

\normalsize
\baselineskip=17pt

\subsection*{Acknowledgements}
The second author acknowledges the Council of Scientific and
Industrial Research INDIA for a research fellowship.


\begin{thebibliography}{HD}






\normalsize
\baselineskip=17pt




\bibitem{BG} Bayart, F.; Grivaux, S.:
\emph{Frequently hypercyclic operators} Trans. Amer. Math. Soc.
{358}(2006), 5083-5117.

\bibitem{BM} Bayart, F.; Matheron, E.:
\emph{Dynamics of Linear Operators} Cambridge Univ. Press 179(2009).


\bibitem{BFP} Bonet, J., Martinez-Gimenez, F., Peris, A.: \emph{Universal and chaotic multipliers on spaces of
operators} J. Math. Anal. Appl 297 (2004) no. 2, 599-611.

\bibitem{BGE} Bonilla, A.; Grosse-Erdmann, K-G.:
\emph{Frequently hypercyclic operators and vectors}, Ergodic Theory
Dynam. Systems 27(2007), 383-404. Erratum: Ergodic Theory Dynam. Systems 29 (2009), 1993-1994.

\bibitem{BE} Bonilla, A., Grosse-Erdmann, K-G.: \emph{Frequently hypercyclic
subspaces} Monatsh. Math. 168 (2012) no. 3-4, 305-320.

\bibitem{AC} Carbery, A.: \emph{Almost-orthogonality in the Schatten-von Neumann classes} J. Operator Theory. 62 (2009),
no.1, 151-158.

\bibitem{C} Chan, K.: \emph{Hypercyclicity of the operator
algebra for a separable Hilbert space} J. Operator Theory 42 (1999)
no. 2, 231-244.

\bibitem{CT} Chan, K.; Taylor, R.: \emph{Hypercyclic subspaces of a Banach
space}, Integral Equations Operator Theory 41 (2001), no. 4,
381-388.

\bibitem{DE} Delaubenfels, R.; Emamirad, H.: \emph{Chaos for
functions of discrete and continuous weighted shift operators}
Ergodic theory and Dynamical Systems 21, 1411-1427 (2001).

\bibitem{DJT} Diestal, J.; Jarcho, H.; Tonge, A.:
\emph{Absolutely Summing Operators} Vol. 43, Cambridge Studies in
Advanced Mathematics, Cambridge Univ. Press, 1995.

\bibitem{BST} Fabian, M.; Habala, P.; Hajek, P.; Montesinos, V.; Zizler, V.:
 \emph{Banach Space Theory: The Basis for Linear and
Non-linear Analysis} Springer, New York, 2011.

\bibitem{Grivaux} Grivaux, S.: \emph{A new class of
frequently hypercyclic operators with applications} Indiana Univ.
Math. J. 60 (2011), no. 4, 1177-1201.

\bibitem{PE} Grosse-Erdmann, K-G.; Peris, A.: \emph{Linear
Chaos} Universitext. Springer, London, 2011.

\bibitem{M-M} Gupta,, M.; Mundayadan, A.: \emph{$q$-Frequently hypercyclic operators} Banach J. Math. Anal., vol. 9, no. 2 (2015) 114-126.

\bibitem{KG} Kamthan, P.K; Gupta, M.: \emph{Sequence Spaces and Series}, Marcell-Dekker, 65 (1981).

\bibitem{LT} Lindenstrauss, J.; Tzaffiri, T.: \emph{Classical Banach Spaces} Springer-Verlag, 1997.

\bibitem{FP} Martinez-Gimenez, F.; Peris, A.: \emph{Universality and chaos for tensor produtcs of
operators} J. Approx. Theory 124 (2003) no.1, 7-24.

\bibitem{Oho} McArthur, C.W.; Retherford, J.R.: \emph{Some applications of an inequality in locally convex spaces} Trans. Amer. Math. Soc., 137 (1969), 115-123.

\bibitem{H} Petersson, H.: \emph{Hypercyclic conjugate operators}
Integral Equations Operator Theory 57 (2007), 413-423.


\end{thebibliography}
\end{document}